\documentclass{articlefederico2017}

\usepackage{amsfonts}
\usepackage{amsmath}
\usepackage{amssymb}
\usepackage{amsthm}
\usepackage{upref}
\usepackage[colorlinks,final,backref=page,hyperindex]{hyperref}
\usepackage{mathtools}
\usepackage{float}

\newcommand\hc{\operatorname{hc}}
\newcommand\tc{\operatorname{tc}}
\newcommand\lc{\operatorname{lc}}
\newcommand\pc{\operatorname{pc}}

\newtheorem{theorem}{Theorem}[section] \numberwithin{equation}{section}

\newtheorem{proposition}[theorem]{Proposition}

\newtheorem{exam}{Example}

\renewcommand\emptyset{\varnothing}

\newcommand\commentout[1]{}
\newcommand\Def[1]{{\bf #1}}

\newcommand\vol{\operatorname{vol}}

\newcommand\ehr{\operatorname{ehr}}

\newcommand\aff{\operatorname{aff}} 
\newcommand\thespan{\operatorname{span}} 

\newcommand\ZZ{\mathbb{Z}}
\newcommand\QQ{\mathbb{Q}}
\newcommand\RR{\mathbb{R}}

\newcommand\cE{\mathcal{E}}
\newcommand\cF{\mathcal{F}}

\newcommand\Zono{\mathcal{Z}}

\newcommand\bA{A}
\newcommand\bB{B}
\newcommand\bC{C}
\newcommand\bD{D}
\newcommand\bI{\mathbf{I}}
\newcommand\bU{\mathbf{U}}
\newcommand\bW{\mathbf{W}}
\newcommand\ba{\mathbf{a}}
\newcommand\bb{\mathbf{b}}
\newcommand\bc{\mathbf{c}}
\newcommand\be{\mathsf{e}}

\newcommand\bs{\mathbf{s}}
\newcommand\bt{\mathbf{t}}
\newcommand\bu{\mathbf{u}}
\newcommand\bv{\mathbf{v}}
\newcommand\bw{\mathbf{w}}

\newcommand\bzero{\mathbf{0}}
\newcommand\bone{\mathbf{1}}
\newcommand\blambda{\mathbf{\lambda}}

\newcommand\hobox{
\unitlength .23 mm 
\begin{picture}(10,10)(0,0)
\linethickness{0.2mm}
\put(0,-0.4){\line(0,1){10.83}}
\linethickness{0.2mm}
\put(-0.4,0){\line(1,0){10.83}}
\linethickness{0.2mm}
\multiput(10,0.11)(0,1.885){6}{\line(0,1){0.9}}
\linethickness{0.2mm}
\multiput(0.11,10)(1.885,0){6}{\line(1,0){0.9}}
\end{picture}
\,
}

\begin{document}

\title{The Arithmetic of Coxeter Permutahedra}

\author{
\textsf{Federico Ardila\footnote{\noindent \textsf{San Francisco State University, Universidad de Los Andes; federico@sfsu.edu.}}}
\and
\textsf{Matthias Beck\footnote{\noindent \textsf{San Francisco State University, Freie Universit\"at Berlin; mattbeck@sfsu.edu.}}}
\and
\textsf{Jodi McWhirter\footnote{\noindent \textsf{Washington University in St.~Louis; jodi.mcwhirter@wustl.edu \newline
FA was supported by National Science Foundation grant DMS-1855610 and Simons Fellowship 613384.}}}
}

\maketitle

\begin{abstract}
Ehrhart theory measures a polytope $P$ discretely by counting the lattice points inside its dilates $P, 2P, 3P, \ldots$. We compute the Ehrhart quasipolynomials of the standard Coxeter permutahedra for the classical Coxeter groups, expressing them in terms of the Lambert $W$ function. A central tool is a description of the Ehrhart theory of a rational translate of an integer zonotope.
\end{abstract}


\section{Introduction}

\subsection{Measuring combinatorial polytopes}

Measuring is one of the central questions in mathematics: How do we quantify the size or complexity of a mathematical object? In the theory of polytopes, it is natural to measure a shape by means of its volume or its surface area. 
Computing these quantities for a high-dimensional polytope $P$ is a difficult task
\cite{baranyfuredi, dyerfrieze}, and one approach has been to discretize the question. One
places the polytope $P$ on a grid and asks: How many grid points  does $P$ contain?
How many grid points do its dilates $2P, 3P, 4P, \ldots$ contain? This approach is illustrated in Figure \ref{root permutahedra} for four polygons. 

Ehrhart
\cite{ehrhartpolynomial} showed that when the polytope $P$ has integer (or rational) vertices,
then there is a polynomial (or quasipolynomial) $\ehr_P(x)$ such that the dilate $tP$
contains exactly $\ehr_P(t)$ grid points for any positive integer~$t$. He also showed that
the leading coefficient of $\ehr_P(x)$ equals the (suitably normalized) volume of $P$, and
the second leading coefficient equals half of the (suitably normalized) surface area.
Therefore the \emph{Ehrhart (quasi)polynomial} (which we will define in detail in Section~\ref{sec:Ehrhart} below) is a more precise measure of size than these two
quantities. Ehrhart theory is devoted to measuring polytopes in this way, computing
continuous quantities discretely (see, e.g.,~\cite{cdd}).

\begin{figure}[h]
	\centering
	\includegraphics[height=1.2in]{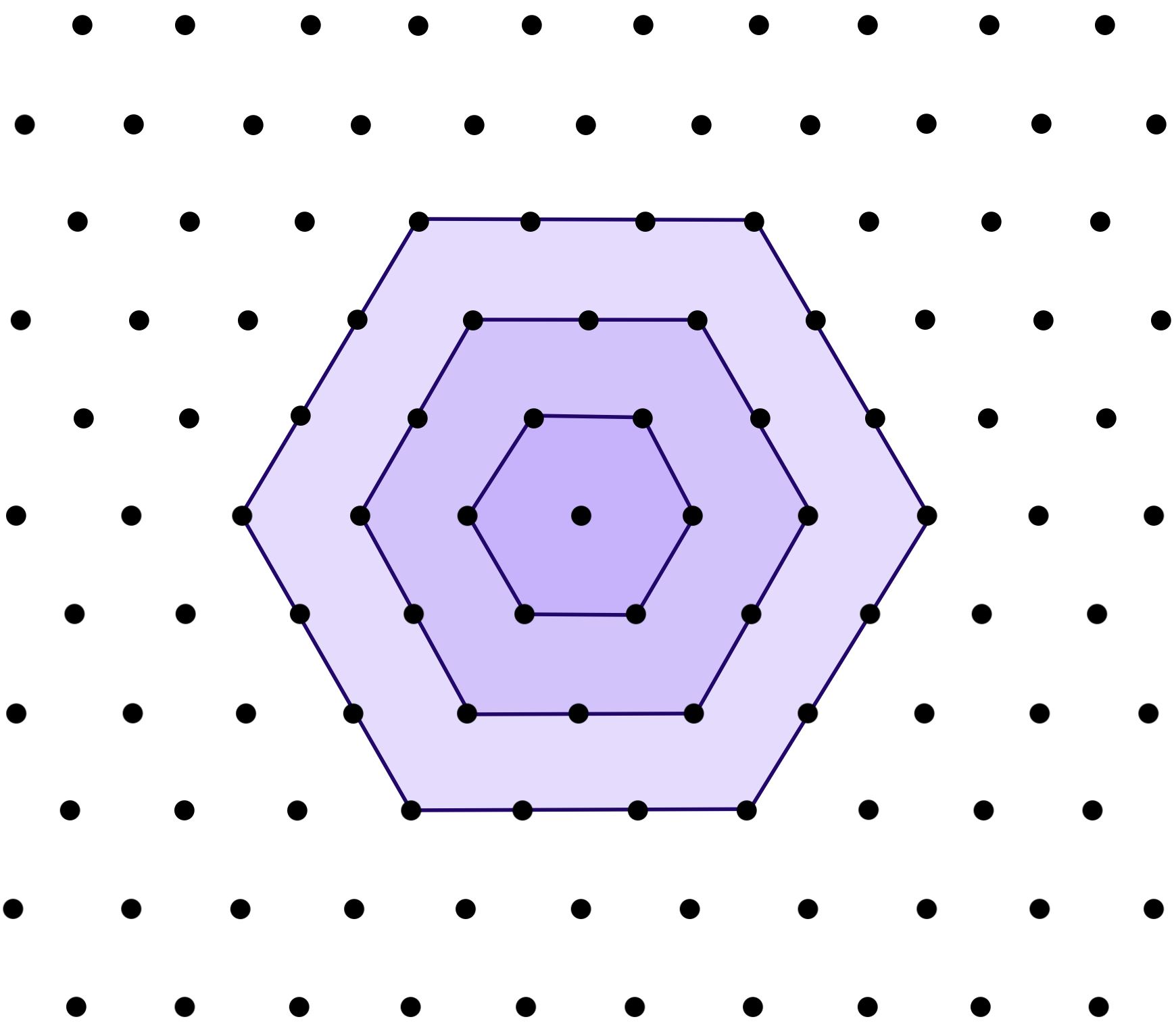}
	\hspace{.3in}
	\includegraphics[height=1.2in]{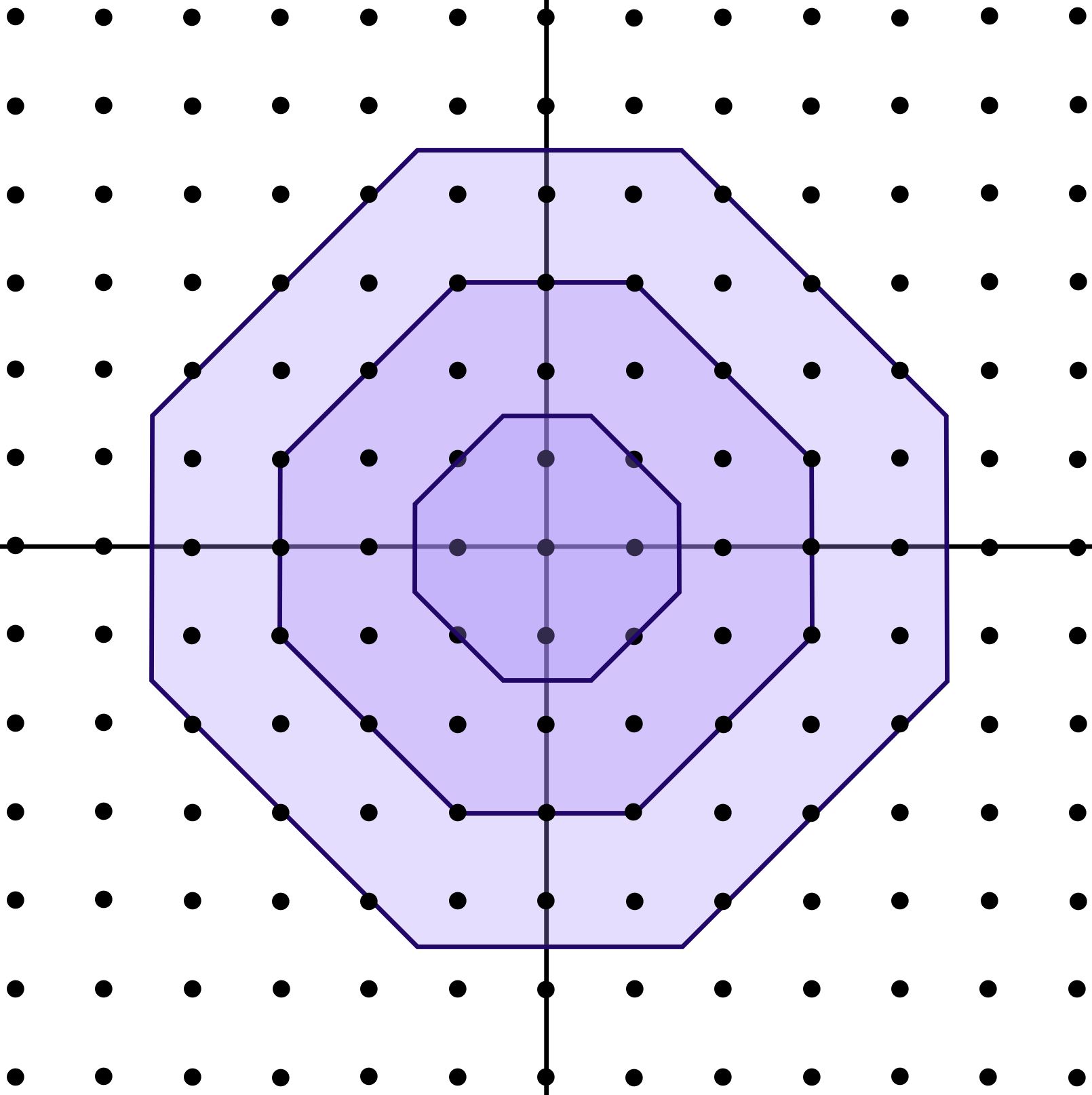}
	\hspace{.3in}
	\includegraphics[height=1.2in]{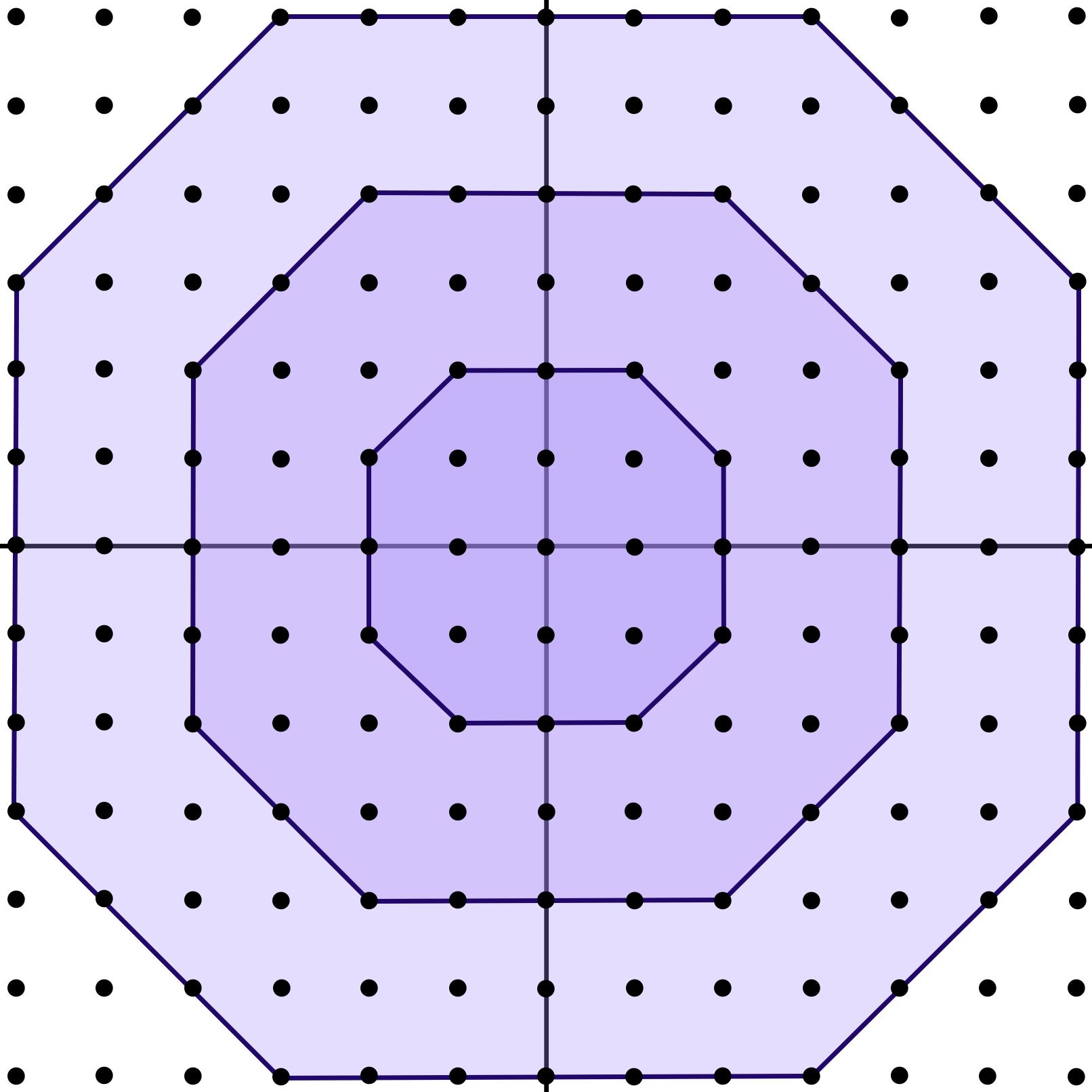}
	\hspace{.3in}
	\includegraphics[height=1.2in]{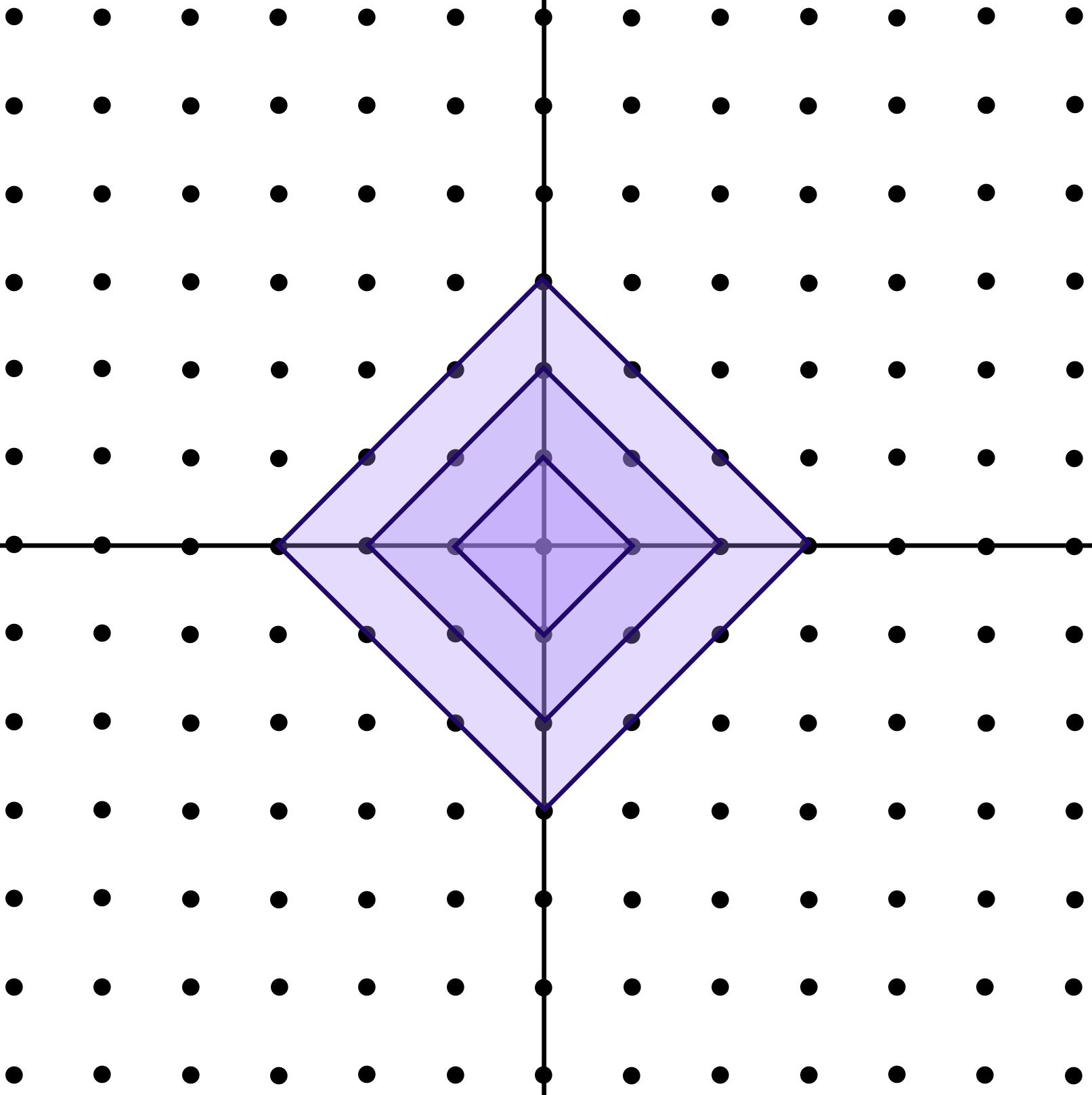}
		\caption{The first three dilates of the standard Coxeter permutahedra $\Pi({\bA_2}), \Pi({\bB_2}), \Pi({\bC_2}),$ and $\Pi({\bD_2})$. Their $t$th dilates contain $1+3t+3t^2$, ($1+4t+7t^2$ for $t$ even and $2t+7t^2$ for $t$ odd), $1+6t+14t^2$, and $1+2t+2t^2$ lattice points, respectively.}
	\label{root permutahedra}
\end{figure}

Combinatorics studies the possibilities of a discrete situation; for example, the possible
ways of reordering, or \Def{permuting} the numbers $1, \ldots, n$. In most situations of
interest, the number of possibilities of a discrete problem is tremendously large, so one
needs to find intelligent ways of organizing them. Geometric combinatorics offers an
approach: model the (discrete) possibilities of a problem with a (continuous) polytope. A
classic example is the \Def{permutahedron} $\Pi_n$, a polytope whose vertices are the $n!$
permutations of $\{1,2,\ldots,n\}$. (Figure \ref{fig:perm} shows the permutahedron
$\Pi_4$.) One can answer many questions about permutations using the geometry of this polytope. In this way, the general strategy of  geometric combinatorics is to model discrete problems continuously.

\begin{figure}[h]
	\centering
	\includegraphics[height=2in]{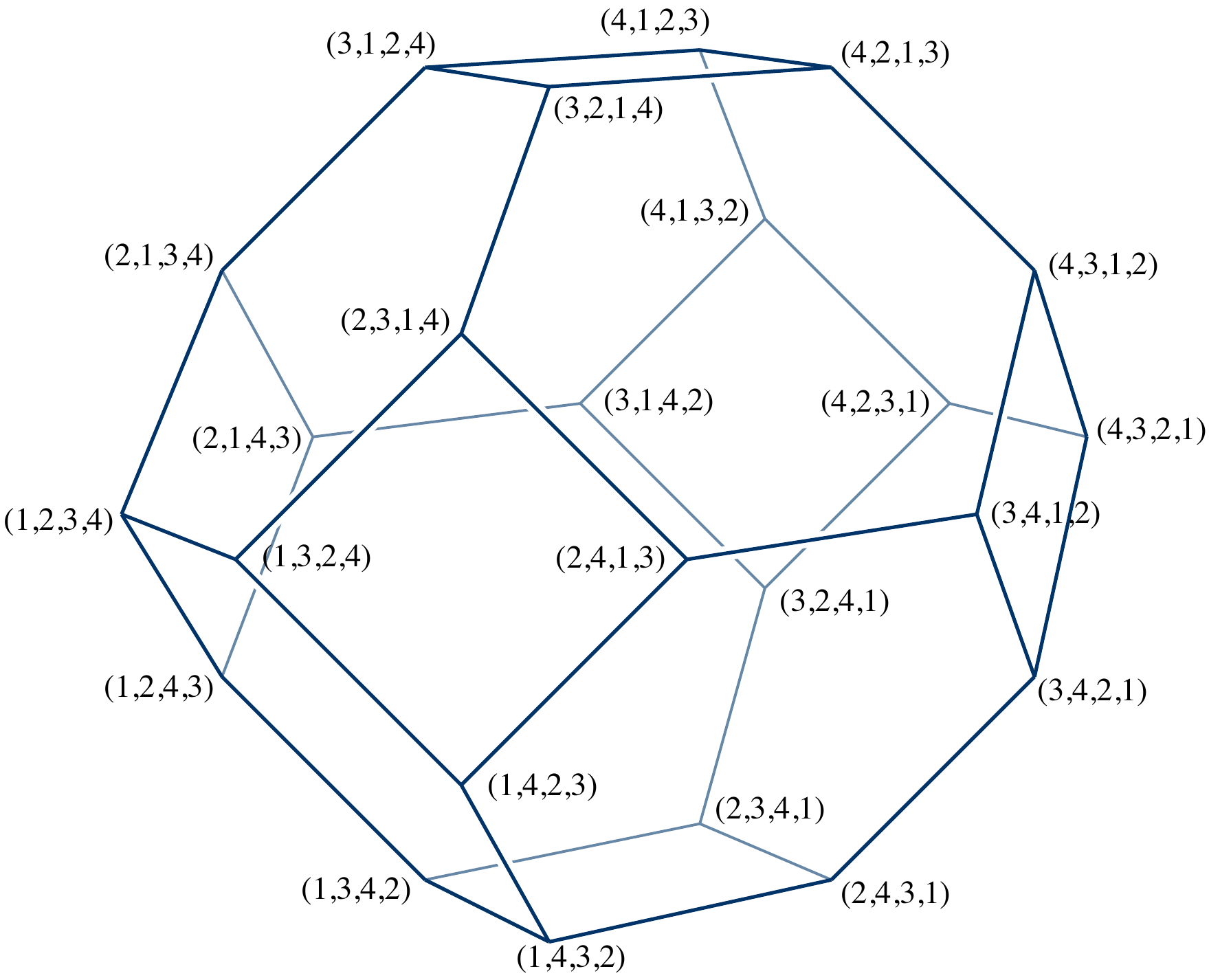}
		\caption{The permutahedron $\Pi_4$ organizes the 24 permutations of $\{1,2,3,4\}$.}
	\label{fig:perm}
\end{figure}

Combining these two forms of interplay between the discrete and the continuous, it is natural to begin with a discrete problem, model it in terms of a continuous polytope, and then measure that polytope discretely. Stanley \cite{stanleyzonotopegraphicaldegree} pioneered this line of inquiry, with the following beautiful theorem.

\begin{theorem}[Stanley \cite{stanleyzonotopegraphicaldegree}] \label{th:stanley}
The Ehrhart polynomial of the permutahedron $\Pi_n$ is 
\[
\ehr_{\Pi_n}(t) \ = \ a_{n-1}t^{n-1} + a_{n-2}t^{n-2} + \cdots + a_1t + a_0 \, ,
\]
where $a_i$ is the number of graphs with $i$ edges on the vertices $\{1, \ldots, n\}$ that contain no cycles. In particular, the normalized volume of the permutahedron $\Pi_n$ is the number of trees on $\{1, \ldots, n\}$, which equals $n^{n-2}$.
\end{theorem}


\subsection{Our results: measuring classical Coxeter permutahedra}\label{sec:ourresults}

The permutahedron $\Pi_n$ is one of an important family of highly symmetric polytopes: the
reduced, crystallographic \Def{standard Coxeter permutahedra}; see Section \ref{sec:Lie} for a precise definition and some Lie theoretic context.
These polytopes come in four infinite families $A_{n-1}, B_n, C_n, D_n$  $(n \geq 1)$ called the \Def{classical types}, and five exceptions $E_6, E_7, E_8, F_4,$ and $G_2$. 
The standard Coxeter permutahedra of the classical types are the following polytopes in $\RR^n$:
  \begin{eqnarray*}
    \Pi({\bA_{n-1}}) &:=& \text{conv}\{ \text{permutations of } \tfrac12(-n+1, -n+3, \ldots, n-3, n-1) \}, \\
    \Pi({\bB_{n}}) &:=& \text{conv}\{\text{signed permutations of } \tfrac 1 2 (1, 3, \dots, 2n-1)\}, \\
    \Pi({\bC_{n}}) &:=& \text{conv}\{\text{signed permutations of } (1, 2, \dots, n) \}, \\
    \Pi({\bD_{n}}) &:=& \text{conv}\{\text{evenly signed permutations of } (0, 1, \dots, n-1)\}.
  \end{eqnarray*}
Here a \Def{signed permutation} of a sequence $S$ is obtained from a permutation of $S$ by introducing signs to the entries arbitrarily; the \Def{evenly signed permutations} are those that introduce an even number of minus signs. Figure \ref{root permutahedra} shows the standard Coxeter permutahedra $\Pi({\bA_2}), \Pi({\bB_2}), \Pi({\bC_2}),$ and $\Pi({\bD_2})$, as well as their second and third dilates.

The goal of this paper is to understand the Ehrhart theory of these four families of polytopes. Our main results are the following. Theorem \ref{th:Ehrhart} generalizes Stanley's Theorem \ref{th:stanley}, offering combinatorial formulas for the Ehrhart quasipolynomials of the Coxeter permutahedra $\Pi({\bA_{n-1}}), \Pi({\bB_{n}}), \Pi({\bC_{n}})$, and $\Pi({\bD_{n}})$ in terms of the combinatorics of forests. Theorems \ref{th:egfZono} and \ref{th:egfCZono}
then give explicit formulas: they compute the exponential generating functions of those
Ehrhart quasipolynomials, in terms of the Lambert $W$ function. Proposition
\ref{prop:zonotope} is an intermediate step that may be of independent interest: it describes
the Ehrhart theory of a rational translate of an integral zonotope. This result was used in \cite{ardilasupinavindas} to compute the equivariant Ehrhart theory of the permutahedron.

We remark that each of these zonotopes can be translated to become an integral polytope, and the Ehrhart polynomials of these integral translates were computed in~\cite{ardilacastillohenley}; 
see also \cite{deconciniprocesizonotope,dezamanoussakisonn} for related work.


\section{Preliminaries}

\subsection{Ehrhart theory}\label{sec:Ehrhart}

A \Def{rational polytope} $P \subset \RR^d$ is the convex hull of finitely many points in $\QQ^d$. We define
\[
  \ehr_P(t) \ := \ \left| tP \cap \ZZ^d \right| ,
\]
for positive integers $t$.
Ehrhart~\cite{ehrhartpolynomial} famously proved that this lattice-point counting function
evaluates to a \Def{quasipolynomial} in $t$, that is,
\[
  \ehr_P(t) \ = \ c_d(t) \, t^d + c_{d-1}(t) \, t^{d-1} + c_0(t)
\]
where $c_0(t), \dots, c_d(t) : \ZZ \rightarrow \QQ$ are periodic functions in $t$; their minimal common period is
the \Def{period} of $\ehr_P(t)$.
Ehrhart also proved that the period of $\ehr_P(t)$ divides the least common multiple of the
denominators of the vertex coordinates of $P$.
In particular, if $P$ is an \emph{integral} polytope, then $\ehr_P(t)$ is a polynomial.

All the polytopes we will consider in this paper are half integral. Therefore the periods of their Ehrhart quasipolynomials will be either 1 or~2. For more on Ehrhart quasipolynomials, see, e.g.,~\cite{ccd}.

\subsection{Zonotopes}\label{sec:zonotopes}

A \Def{zonotope} is the Minkowski sum $\Zono(A)$ of a finite set of line segments $A = \{[\ba_1, \bb_1], \ldots, [\ba_n, \bb_n]\}$ in $\RR^d$; that
is,
\begin{eqnarray*}
  \Zono(A) &:=& \sum_{j =1}^n [\ba_j, \bb_j] \\
  &=& \Big\{\sum_{j =1}^n \bc_j \, : \, \bc_j \in  [\ba_j, \bb_j] \text{ for } 1 \leq j \leq
n \Big\} .
\end{eqnarray*}
For a finite set of vectors $\bU \subset \RR^d$ we define 
\[
\Zono(\bU) \ := \ \sum_{\bu \in \bU} [\bzero, \bu] \, .
\]

Shephard~\cite{shephardzonotopes} showed that the zonotope $\Zono(A)$ may be decomposed as a disjoint union of translates of the half-open parallelepipeds 
\[
\hobox \, \bI \ := \ \sum_{\bu \in \bI} [\bzero, \bu)
\]
spanned by the linearly independent subsets $\bI$ of $\{ \bb_j - \ba_j : 1 \leq j \leq n\}$. This decomposition contains exactly one parallelepiped for each independent subset.
Figure~\ref{u3} displays such a \Def{zonotopal decomposition} of a hexagon.
\begin{figure}[h!]
	\centering
	\includegraphics[width=2in]{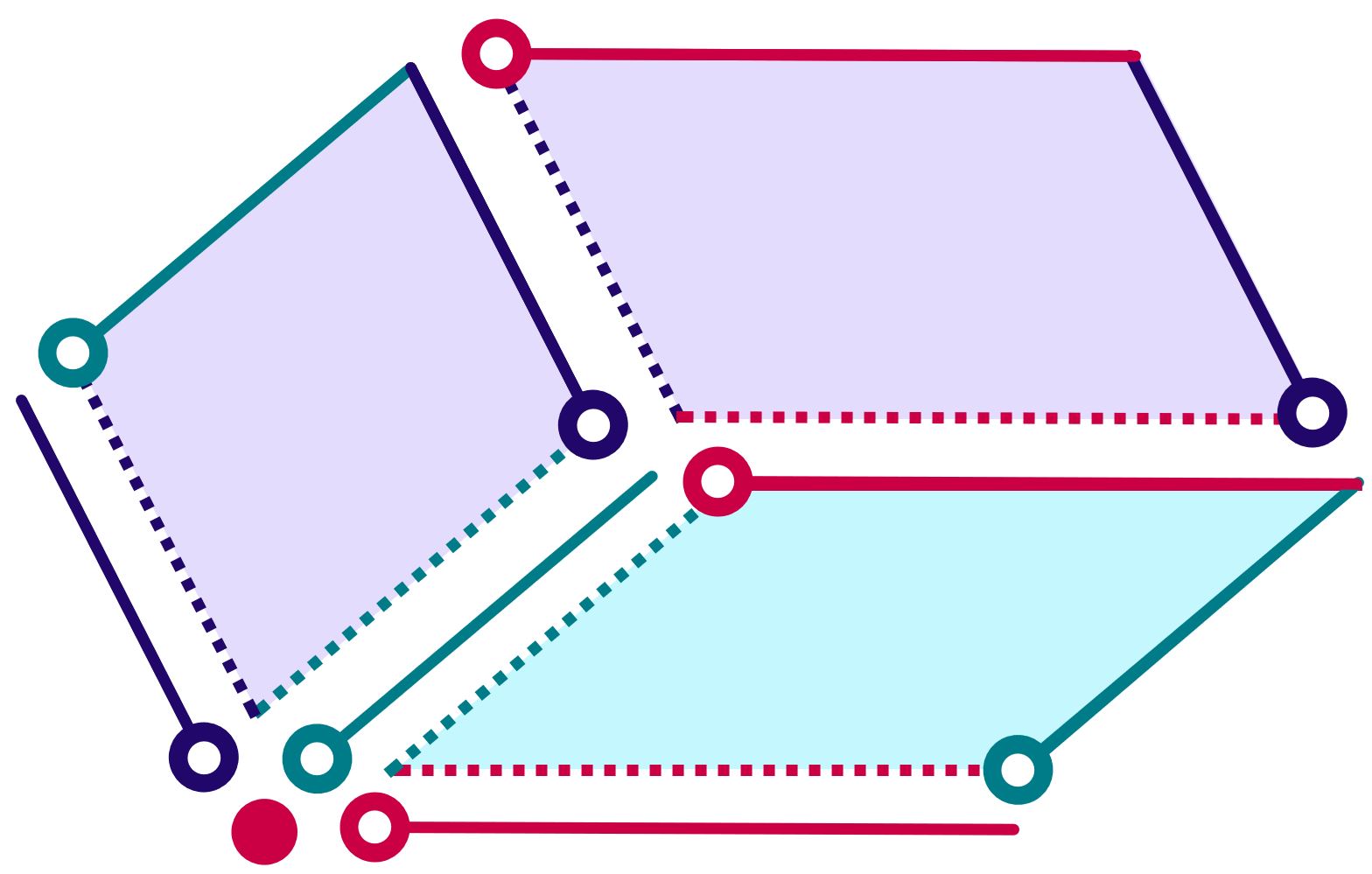}
	\caption{A decomposition of a hexagon into half-open parallelepipeds.}
	\label{u3}
\end{figure}

A useful feature of this decomposition is that lattice half-open parallelepipeds are 
arithmetically quite simple: $\hobox \, \bI$ contains exactly $\vol (\hobox \, \bI)$ lattice
points, where $\vol (\hobox \, \bI)$ denotes the \Def{relative volume} of $\hobox \, \bI$, measured with
respect to the sublattice $\ZZ^d \cap \aff (\hobox \, \bI)$ in the affine space spanned by the parallelepiped.
This implies the following result.

\begin{proposition} (Stanley, \cite{stanleyzonotopegraphicaldegree}) \label{prop:Zzonotope}
Let $\bU \subset \ZZ^d$ be a finite set of vectors. Then the Ehrhart polynomial of the integral zonotope $\Zono(\bU)$ is
\[
  \ehr_{\Zono(\bU) } (t) \ = \sum_{ \stackrel{\bW \subseteq \bU}{ \text{ \rm lin.~indep.} }}
 \vol(\bW) \, t^{ |\bW| }
\]
where $|\bW|$ denotes the number of vectors in $\bW$ and $\vol(\bW)$ is the relative volume of the parallelepiped generated by $\bW$.
\end{proposition}

\subsection{Lie combinatorics} \label{sec:Lie}

Assuming familiarity with the combinatorics of Lie theory \cite{humphreysCoxeter} (for this
section only), we briefly explain the geometric origin of the polytopes that are our main
objects of study. 
Finite \Def{root systems} are highly symmetric configurations of vectors that play a central role in many areas of mathematics and physics, such as the classification of regular polytopes \cite{coxeter} and of semisimple Lie groups and Lie algebras \cite{humphreysLie}. The finite crystallographic root systems can be completely classified; they come in four infinite families:
\begin{align*}
\bA_{n-1} \ &:= \ \left\{\pm(\be_i-\be_j) : \, 1\leq i<j\leq n \right\}, \\
\bB_n \ &:= \ \left\{\pm(\be_i-\be_j),~ \pm(\be_i+\be_j) : \, 1\leq i<j\leq n \right\}\cup \left\{\pm \be_i : \, 1\leq i\leq n \right\}, \\
\bC_n \ &:= \ \left\{\pm(\be_i-\be_j),~ \pm(\be_i+\be_j) : \, 1\leq i<j\leq n \right\}\cup \left\{ \pm 2 \, \be_i : \, 1\leq i\leq n \right\}, \\
\bD_n \ &:= \ \left\{\pm(\be_i-\be_j),~ \pm(\be_i+\be_j) : \, 1\leq i<j\leq n \right\}
\end{align*}
and five exceptions: $E_6, E_7, E_8, F_4,$ and $G_2$. For each of the four infinite families $A_n, B_n, C_n, D_n$ of root systems $\Phi$, we can let the \Def{positive roots} $\Phi^+$ be those obtained by choosing the plus sign in each $\pm$ above.

Let $\Phi$ be a finite root system of rank $d$ and $W$ be its Weyl group. Let $\Phi^+ \subset
\Phi$ be a choice of positive roots. The \Def{standard Coxeter permutahedron of $\Phi$} is
the zonotope 
\begin{eqnarray*}
\Pi(\Phi) &:= &\sum_{\alpha \in \Phi^+} \left[- \tfrac \alpha 2, \tfrac \alpha 2 \right] \\
&=& \text{conv} \{ w\cdot \rho : \, w \in W\}
\end{eqnarray*}
where $\rho := \frac12(\sum_{\alpha \in \Phi^+} \alpha)$. 
These polytopes, and their deformations, are fundamental objects in the representation theory of semisimple Lie algebras \cite{humphreysLie}, in many problems in optimization \cite{ardilacastilloeurpostnikov}, and in the combinatorics of (signed) permutations, among other areas.

For the classical root systems $A_{n-1}, B_n, C_n, D_n$, the standard Coxeter permutahedra are precisely the polytopes $\Pi({A_{n-1}}), \Pi({B_n}), \Pi({C_n}), \Pi({D_n})$ introduced in Section~\ref{sec:ourresults}.


\section{Almost integral zonotopes and their Ehrhart theory}

The arithmetic of zonotopes described in Section~\ref{sec:zonotopes} becomes much more subtle when the zonotope is not integral. However, we can still describe it for \Def{almost integral zonotopes} $\bv + \Zono(\bU)$ , which are obtained by translating
an integral zonotope $\Zono(\bU)$ by a rational vector $\bv$. They satisfy  the following analog of Stanley's Proposition \ref{prop:Zzonotope}.

\begin{proposition} \label{prop:zonotope}
Let $\bU \in \ZZ^d$ be a finite set of integer vectors and $\bv \in \QQ^d$ be a rational vector.
Then the Ehrhart quasipolynomial of the almost integral zonotope $\bv + \Zono(\bU)$ equals
\[
  \ehr_{ \bv + \Zono(\bU) } (t) \ = \sum_{\stackrel{\bW \subseteq \bU}{ \text{ \rm lin.~indep.} }}
\chi_\bW(t) \, \vol(\bW) \, t^{ |\bW| }
\]
where 
\[
  \chi_\bW(t) \ := \ \begin{cases}
    1 & \text{ \rm if } (t \bv + \thespan(\bW)) \cap \ZZ^d \ne \emptyset , \\ 
    0 & \text{ \rm otherwise. } 
  \end{cases}
\]
\end{proposition}

\begin{proof}
The zonotope $t(\bv+\Zono(\bU))$ can be subdivided into lattice translates of the half-open parallelepipeds $t(\bv+ \hobox \, \bW)$ for the linearly independent subsets $\bW \subseteq \bU$. Let us count the lattice points in $t(\bv+ \hobox \, \bW)$; there are two cases:

1. If $t \bv + \thespan(\bW)$ does not intersect $\ZZ^d$ then 
$ |t(\bv +  \hobox \, \bW) \cap \ZZ^d| = 0$.

2. If $t \bv + \thespan(\bW)$ contains a lattice point $\bu \in \ZZ^d$, then it also contains the lattice points $\bu + \bw$ for all $\bw \in \bW$, so $\Lambda:= (t \bv + \thespan(\bW)) \cap \ZZ^d$ is a $|\bW|$-dimensional lattice.
Since $t \bv + \thespan(\bW)$ can be tiled by integer translates of the half-open parallelepiped $t(\bv +  \hobox \, \bW)$, and that linear space contains the lattice $\Lambda$,  each tile must contain $\vol (t \cdot \hobox \, \bW)$ lattice points. Therefore
\[
  \left| t(\bv +  \hobox \, \bW) \cap \ZZ^d \right|
  \ = \ \vol (t \cdot \hobox \, \bW)
  \ = \ \vol (\hobox \, \bW) \, t^{|\bW|} 
\]
and the desired result follows. 
\end{proof}


In \cite{ardilasupinavindas}, Proposition \ref{prop:zonotope} is used to describe the equivariant Ehrhart theory of the permutahedron and prove a series of conjectures due to Stapledon \cite{stapledonequivariant} in this special case.

\section{Classical root systems, signed graphs and Ehrhart functions}

We will express the Ehrhart quasipolynomials of the classical Coxeter permutahedra in terms of the combinatorics of signed graphs. These objects originated in the social sciences and have found applications also in biology, physics, computer science, and economics; they are a very useful combinatorial model for the classical root systems. See \cite{zaslavskydynamicsurvey} for a comprehensive bibliography.

\subsection{Signed graphs as a model for classical root systems}

A \Def{signed graph} $G = (\Gamma,\sigma)$ consists of a graph $\Gamma = (V,E)$ and a signature
$\sigma \in \left\{ \pm \right\}^E$. The underlying graph $\Gamma$ may have multiple edges, loops,
\Def{halfedges} (with only one endpoint), and \Def{loose edges} (with no endpoints); the latter two have
no signs. For the applications we have in mind, we may assume that $G$ has no loose edges and no repeated signed edges; we do allow $G$ to have two parallel edges with opposite signs.

A signed graph $G = (\Gamma, \sigma)$ is \Def{balanced} if each cycle has an even number of negative edges.
An unsigned graph can be realized by a signed graph all of whose edges are labelled with $+$; it is automatically balanced.

Continuing a well-established dictionary~\cite{zaslavskygeometryrootsystems}, we encode a subset $S \subseteq \Phi^+$ of positive roots of one of
the classical root systems $\Phi \in \{ A_{n-1}, B_n, C_n, D_n : \, n \geq 1 \}$ in the signed graph $G_S$ on $n$ nodes with

\medskip
\begin{tabular}{lcl}
$\bullet$ a positive edge $ij$ for each $\be_i - \be_j \in S$, \qquad &&
$\bullet$ a halfedge at $j$ for each $\be_j \in S$, and \\
$\bullet$ a negative edge $ij$ for each $\be_i + \be_j \in S$, \qquad &&
$\bullet$ a negative loop at $j$ for each $2\be_j \in S$.
\end{tabular}
\medskip

The \Def{$\Phi$-graphs} are the signed graphs encoding the subsets of $\Phi^+$. More explicitly, a signed graph is an \Def{$A_{n-1}$-graph} (or simply a \Def{graph})  if it contains only positive edges, a \Def{$B_n$}-graph if it contains no loops, a \Def{$C_n$-graph} if it contains no halfedges, and a \Def{$D_n$-graph} if it contains neither halfedges nor loops. For a $\Phi$-graph $G$, we let $\Phi_G \subseteq \Phi^+$ be the corresponding set of positive roots of $\Phi$.

It will be important to understand which subsets of $\Phi^+$ are linearly independent; to this end we make the following definitions.
	\begin{itemize}
		\item A (signed) \Def{tree} is a connected (signed) graph with no cycles, loops, or halfedges.
		\item A (signed) \Def{halfedge-tree}  is a connected (signed) graph with no cycles or loops, and a single halfedge.
		\item A (signed) \Def{loop-tree}  is a connected (signed) graph with no cycles or halfedges, and a single loop.
		\item A (signed) \Def{pseudotree} is a connected (signed) graph with no loops or halfedges that contains a single cycle (which is unbalanced). 
		\item A \Def{signed pseudoforest}  is a signed graph whose connected components are signed trees, signed halfedge-trees, signed loop-trees, or signed pseudotrees.
		\item A \Def{$\Phi$-forest} is a signed pseudoforest that is a $\Phi$-graph  for $\Phi \in \{A_{n-1}, B_n, C_n, D_n : \, n \geq 1\}$.
		\item A \Def{$\Phi$-tree} is a connected $\Phi$-forest for $\Phi \in \{A_{n-1}, B_n, C_n, D_n : \, n \geq 1\}$.
	\end{itemize} 

\noindent In particular the $A_{n-1}$-pseudoforests are the forests on $[n] := \{ 1, 2, \dots, n \}$. For a signed pseudoforest $G$, we let $\tc(G)$, $\hc(G)$, $\lc(G)$, and $\pc(G)$ be the number of tree components, halfedge-tree components, loop-tree components, and pseudotree components, respectively.

In this language, we recall and expand on results by Zaslavsky~\cite{zaslavskysignedgraphs} and Ardila--Castillo--Henley \cite{ardilacastillohenley} on the arithmetic matroids of the classical root systems.
Recall that for a linearly independent set $\bW \subset \ZZ^n$, we write $\vol(\bW)$ for the
relative volume of the parallelepiped $\Zono(\bW)$ generated by $\bW$.

\begin{proposition} \label{prop:arithmeticmatroid}
 \cite{ardilacastillohenley, zaslavskysignedgraphs}
Let $\Phi \in \{ A_{n-1}, B_n, C_n, D_n \}$ be a root system. The independent subsets of $\Phi^+$ are the sets $\Phi_G$ 
for the $\Phi$-forests $G$ on $[n]$. For each such $G$,
\[
  |\Phi_G| = n - \tc(G)
  \qquad \text{ and } \qquad
  \vol (\Phi_G) = 2^{\pc(G) + \lc(G)}.
\]
\end{proposition}


\subsection{Ehrhart quasipolynomials of standard Coxeter permutahedron of classical type} 

We also define the \Def{integral Coxeter permutahedron}
\[
\Pi^\ZZ(\Phi) \ := \ \sum_{\alpha \in \Phi^+} [0,\alpha].
\]
This is a translate of the standard Coxeter permutahedron $\Pi(\Phi)$ which is an integral polytope for all $\Phi$. Its Ehrhart theory was computed in \cite{ardilacastillohenley}. This is sometimes, but not always, the same as the Ehrhart theory of $\Pi(\Phi)$, as we will see in this section, particularly in Theorem~\ref{th:Ehrhart}.

It follows from the description in Section \ref{sec:ourresults} that the standard Coxeter permutahedron $\Pi(\Phi)$ is an integral polytope precisely for $\Phi \in  \{A_{n-1} : \, n\geq 1 \text{ odd}\} \cup \{C_n : \, n \geq 1\} \cup \{D_n : \, n \geq 1\}$.
It is shifted $\frac 1 2 \bone := \frac 1 2 (\be_1 + \dots + \be_n)$ away from being integral for $\Phi \in \{A_n : \, n \geq 2 \text{ even}\} \cup \{B_n : \, n \geq 1\}$.

\newpage

\begin{proposition}\label{prop:latticeplanes}
Let $\Phi \in \{A_n : \, n \geq 2 \text{ even}\} \cup \{B_n : \, n \geq 1\}$.
For a $\Phi$-forest $G$, the affine subspace $\frac 1 2 \bone + \thespan(\Phi_G)$ contains lattice points if and only if every (signed or unsigned) tree component of $G$ has an even number of vertices.
\end{proposition} 

\begin{proof}
Let $G_1, \ldots, G_k$ be the connected components of $G$, on vertex sets $V_1, \ldots, V_k$,
respectively. Along the decomposition $\RR^n = \RR^{V_1} \oplus \cdots \oplus \RR^{V_k}$, we have
\[
\tfrac 1 2 \bone + \thespan(\Phi_G) \ = \
\sum_{i=1}^k \tfrac 1 2 \bone_{V_i} + \thespan(\Phi_{G_i})
\]
where $\bone_{V} := \sum_{i \in V} \be_i$ for $V \subseteq [n]$. 
Therefore $\tfrac 1 2 \bone + \thespan(\Phi_G)$ contains a lattice point in $\ZZ^n$  if and only if $\tfrac 1 2 \bone_{V_i} + \thespan(\Phi_{G_i})$ contains a lattice point in $\ZZ^{V_i}$ for every $1 \leq i \leq k$. For this reason, it suffices to prove the proposition for $\Phi$-trees. 

For every labeling $\lambda \in \RR^{E(G)}$ of the edges of $G$ with scalars, we will write
\begin{equation} \label{eq:lambda}
 \bv_G(\blambda) \ := \ \tfrac 1 2 \bone + \sum_{ \bs \in E(G) } \lambda_{ \bs } \, \bs \, . 
\end{equation}
We need to show that for a $\Phi$-tree $G$, there exists $\lambda \in \RR^{E(G)}$ with $\bv_G(\blambda) \in \ZZ^n$ if and only if $G$ is not a (signed or unsigned) tree with an odd number of vertices. 
We proceed by cases.

\medskip

\noindent (i) \textbf{Trees}: Let $G = ([n],E)$ be a tree. If 
\begin{equation} \label{eq:lambdaA}
 \bv_G(\blambda) \ := \ \tfrac 1 2 \bone + \sum_{ ij \in E(G) } \lambda_{ ij } \left( \be_i - \be_j \right)
\end{equation}
is a lattice point for some choice of scalars $\blambda =( \lambda_{ ij })_{ij \in E}$,
then the sum of the coordinates of  $\bv_G(\lambda)$---which ought to be an integer---equals $\tfrac 1 2 n$. Therefore $n$ is even.

Conversely, suppose $n$ is even. For each edge $e=ij$ of $G$, let
\[
\lambda_{ij} = \begin{cases} 
0 & \text{if $G-e$ consists of two subgraphs with an even number of vertices each, and} \\
\frac12 & \text{if $G-e$ consists of two subgraphs with an odd number of vertices each}.
\end{cases}
\]
We claim that $\bv_G(\blambda)$, as defined in \eqref{eq:lambdaA}, is an integer vector. To see this, consider any vertex $1 \leq m \leq n$ and suppose that when we remove $m$ and its adjacent edges, we are left with subtrees with vertex sets $V_1, \ldots, V_k$. Then
\[
 \bv_G(\blambda)_m \ \equiv \ \tfrac12 + \tfrac12(\text{number of $1 \leq i \leq k$ such that $|V_i|$ is odd})  \quad (\text{mod } 1),
\]
and this is an integer since $\sum_{i=1}^k |V_i| = n-1$ is odd. 

We conclude that for a tree $G$, the affine subspace $\tfrac 1 2 \bone + \thespan(\Phi_G)$ contains lattice points if and only if $G$ has an even number of vertices, as desired.

%

\medskip

\noindent (ii) \textbf{Signed trees}: Given a subset $S \subseteq B_n = \{\pm \be_i \pm \be_j \, : \, 1 \leq i < j \leq n\} \cup \{\pm \be_i \, : \, 1 \leq i \leq n\}$, we define the \Def{vertex switching} $S_m$ of $S$ at a vertex $1 \leq m \leq n$ to be obtained by changing the sign of each occurrence of $\be_m$ in an element of $S$. Notice that the effect of this transformation on the expression
\[
\tfrac 1 2 \bone + \sum_{ \bs \in S } \lambda_{ \bs } \, \bs 
\]
is simply to change the $m$th coordinate from $\tfrac 12 + a$ to $\tfrac 12 - a$; this does not affect integrality.

Similarly, define the \Def{edge switching} $S_{\bb}$ of $S$ at $\bb \in S$ to be obtained by changing the sign of $\bb$ in $S$. Notice that 
\[
\tfrac 1 2 \bone + \sum_{ \bs \in S } \lambda_{ \bs } \, \bs \ = \
\tfrac 1 2 \bone + \sum_{ \bs \in S_{\bb} } \lambda'_{ \bs } \, \bs 
\]
where $\lambda'$ is obtained from $\lambda$ by switching the sign of $\lambda_{\bs}$.

We conclude that vertex and edge switching a subset $S \subseteq B_n$ does not affect
whether $\tfrac12 \bone + \thespan(S)$ intersects the lattice $\ZZ^n$. Now, it is
known~\cite{zaslavskysignedgraphs} that for any balanced signed graph $G$ there is an
ordinary graph $H$ such that $\Phi_G$ can be obtained from $\Phi_H$ by vertex and edge
switching. In particular---as can also be checked directly---any signed tree $G$ can be turned into an unsigned tree $H$ in this way. Invoking case (i) for the tree $H$, we conclude that for a signed tree $G$, $\tfrac 1 2 \bone + \thespan(\Phi_G)$ contains lattice points if and only if $G$ has an even number of vertices.

\medskip

\noindent (iii) \textbf{Signed halfedge-trees}: Let $G$ be a signed halfedge tree. We need to show that $\tfrac 12 \bone + \thespan(\Phi_G)$ contains a lattice point.
Let $h$ be the halfedge. There are two cases:

\smallskip

a. If $n$ is even, we can label the edges $\bs$ of $G^- := G-h$ with scalars $\lambda_{\bs}$ such that
$\bv_{G^-}(\lambda|_{G^-}) \in \ZZ^{n}$, in view of (ii).
Setting the weight of the halfedge $\lambda_{h} = 0$ we obtain $\bv_{G}(\lambda|_{G}) = \bv_{G^-}(\lambda|_{G^-}) \in \ZZ^n$, as desired.

\smallskip

b. If $n$ is odd, let $G^+$ be the signed tree obtained by turning the halfedge $h$ into a full edge $h^+$, going to a new vertex $n+1$. 
Using (ii), we can label the edges $\bs$ of $G+$ with scalars $\lambda_{\bs}$ such that $\bv_{G^+}(\lambda|_{G^+}) \in \ZZ^{n+1}$.
Setting the weight of the halfedge $h$ in $G$ to be $\lambda_{h} = \lambda_{h^+}$, we
obtain that $\bv_{G}(\lambda|_{G})$ is obtained from $\bv_{G^+}(\lambda|_{G^+})$ by dropping the last coordinate; therefore $\bv_{G}(\lambda|_{G}) \in \ZZ^{n}$ as desired.

\medskip

\noindent (iv) \textbf{Signed pseudotrees}: Let $G$ be a signed pseudotree. We need to find scalars $\lambda_{\bs}$ such that
$ \bv_G(\blambda)$ is a lattice vector.
 Assume, without loss of generality, that its unique (unbalanced) cycle $C$ is formed by the vertices $1, \ldots, m$ in that order. Let $T_1, \ldots, T_k$ be the subtrees of $G$ hanging from cycle $C$; say $T_i$ is rooted at the vertex $a_i$, where $1 \leq a_i \leq m$, and let $\bs_i$ be the edge of $T_i$ connected to $a_i$. We find the scalars $\lambda_s$ in three steps.

\smallskip

1. Thanks to (ii), for each tree $T_i$ with an even number of vertices, we can label its edges $\bs$ with scalars $\lambda_{\bs}$ such that 
\[
\bv_{T_i}(\lambda|_{T_i})  \in \ZZ^{V_i}.
\]

\smallskip

2. For each tree $T_i$ with an odd number of vertices, we can label the edges $\bs$ of $T_i - \bs_i$ with scalars $\lambda_{\bs}$ such that
$\bv_{T_i-\bs_i}(\lambda|_{T_i-\bs_i}) = \tfrac 1 2 \bone_{V_i-a_i} + \sum_{ \bs \in E(T_i)-\bs_i
} \lambda_{ \bs } \, \bs \in \ZZ^{V_i-a_i}$. Setting $\lambda_{\bs_i}=0$, we obtain
\[
\bv_{T_i}(\lambda|_{T_i})  \,  \in \,  (\tfrac12 \be_{a_i} + \ZZ^{V_i}).
\]

\smallskip

3. It remains to choose the scalars $\lambda_{12}, \ldots, \lambda_{m1}$ corresponding to the edges of the cycle $C$. Since $E(G)$ is the disjoint union of $E(C)$ and the $E(T_i)$s, we have
\[
\bv_G(\lambda) = 
\bv_C(\lambda|_C) + \sum_{i=1}^k \bv_{T_i}(\lambda|_{T_i}) + \bu \, ,
\qquad \text{where} \qquad
\bu = \tfrac12\Big(\bone - \bone_{[m]} - \sum_{i=1}^k \bone_{V_i}\Big) \in \RR^m
\]
is supported on the vertices $[m]=\{1, \ldots, m\}$ of the cycle $C$.
Therefore, $\bv_G(\lambda) \in \ZZ^n$ if and only if  we have $\bv_C(\lambda|_C) + \bt \in \ZZ^m$, where
$\bt: = \bu + \tfrac12 \sum_{i \, : \, |V_i| {\text{ even}}} \be_{a_i}$. We rewrite this condition as
\begin{equation}\label{eq:cycle}
\lambda_{12}(e_1-\sigma_1e_2) + 
\lambda_{23}(e_2-\sigma_2e_3) + 
\cdots + 
\lambda_{m1}(e_m-\sigma_me_1)  
+ \bt 
\in \ZZ^m,
\end{equation}
where $\sigma_i$ is the sign of edge connecting $i$ and $i+1$ in $C$; this is equivalent to the following system of equations modulo 1:
\begin{equation}\label{eq:cyclemod1}
\lambda_{12} \equiv  \lambda_{m1}\sigma_m - t_1 , \quad 
\lambda_{23} \equiv  \lambda_{12}\sigma_1 - t_2, \quad 
\dots, \quad 
\lambda_{m1} \equiv \lambda_{m-1,m}\sigma_{m-1} - t_m \qquad (\text{mod } 1).
\end{equation}
Solving for $\lambda_{12}$ gives $\lambda_{12} \equiv \sigma_1\cdots\sigma_m \lambda_{12}+a$ for a scalar $a$. Since the cycle $C$ is unbalanced, $\sigma_1\cdots\sigma_m =-1$, so this equation has the solution $\lambda_{12} \equiv a/2 \,\, (\text{mod }1)$\footnote{In fact it has exactly two solutions $\lambda_{12} \equiv a/2 \,\, (\text{mod }1)$ and  $\lambda_{12} \equiv (1+a)/2 \,\, (\text{mod }1)$, explaining why we have $\vol(\Phi_G) = 2$ in this case.}.
Using \eqref{eq:cyclemod1}, we can then successively compute the values of $\lambda_{23}, \ldots, \lambda_{m1}$, guaranteeing that \eqref{eq:cycle} holds. In turn, this produces a lattice point $\bv_G(\lambda) \in \frac 1 2 \bone + \thespan(\Phi_G)$, as desired.
\end{proof}

\begin{theorem} \label{th:Ehrhart}
Let $\cF(\Phi)$ be the set of $\Phi$-forests, and $\cE(\Phi) \subseteq \cF(\Phi)$ be the set of $\Phi$-forests such that every (signed) tree component has an even number of vertices. 
\begin{enumerate}
\item
The Ehrhart polynomials of the \textbf{integral} Coxeter permutahedra $\Pi^\ZZ(\Phi)$ are 
\[
\ehr_{ \Pi^\ZZ({\Phi})} (t) \ = 
\sum_{G \in \cF(\Phi)} 2^{ \pc(G) +\lc(G)} t^{ n - \tc(G) }.
\]
\item
For $\Phi \in \{A_n : \, n  \geq 2\text{ even}\} \cup \{B_n : \, n \geq 1\}$, the Ehrhart quasipolynomials of the \textbf{standard} Coxeter permutahedra $\Pi(\Phi)$ are 
\[
\ehr_{ \Pi(\Phi)} (t) \ = \ 
 \begin{cases}
\displaystyle\sum_{G \in \cF(\Phi)} 2^{ \pc(G)} t^{ n - \tc(G) }  & \textrm{ if $t$
is even,} \\
\displaystyle \sum_{G \in \cE(\Phi)} 2^{ \pc(G)} t^{ n - \tc(G) }  & \textrm{ if $t$
is odd.}
\end{cases}
\]
For $\Phi \in  \{A_{n-1} : \, n\geq 1 \text{ odd}\} \cup \{C_n : \, n \geq 1\} \cup \{D_n : \, n \geq 1\}$, we have $\ehr_{ \Pi(\Phi)} (t) = \ehr_{ \Pi^\ZZ(\Phi)}(t)$.
\end{enumerate}
\end{theorem}

\begin{proof}
This is the result of applying Proposition \ref{prop:zonotope} to these zonotopes, taking
into account Propositions~\ref{prop:arithmeticmatroid} and~\ref{prop:latticeplanes}, and the fact that $\Phi$-forests of type $A$ and $B$ contain no loop components.
\end{proof}


\section{Explicit formulas: the generating functions}

In this section, we compute the generating functions for the Ehrhart (quasi)polynomials of the Coxeter permutahedra of the classical root systems. We will express them in terms of the \Def{Lambert $\mathbf{W}$ function}
\[
W(x) \ = \ \sum_{n \geq 1} (-n)^{n-1} \frac{x^n}{n!} \, .
\]
As a function of a complex variable $x$, this is the principal branch of the inverse function of $xe^x$. It satisfies
\[
W(x) \, e^{W(x)} \ = \ x \, .
\]
Combinatorially, $-W(-x)$ is the exponential generating function for $r_n=n^{n-1}$, the
number of rooted trees $(T,r)$ on $[n]$, where $T$ is a tree on $[n]$ and $r$ is a special
vertex called the \Def{root}~\cite[Proposition 5.3.2]{stanleyec2}.

To compute the generating functions of the Ehrhart (quasi)polynomials that interest us, we first need some enumerative results on trees.


\subsection{Tree enumeration}

\begin{proposition} \label{prop:trees}
The enumeration of (signed) trees, (signed) pseudotrees, signed halfedge-trees, and signed loop-trees is given by the following formulas.
\begin{enumerate}
\item
The number of trees on $[n]$ is $t_n = n^{n-2}$. The exponential generating function for this sequence is
\[
T(x) \ := \ \sum_{n \geq 1} n^{n-2} \frac{x^n}{n!} \ = \ -W(-x) - \frac12 W(-x)^2.
\]
\item
The number of pseudotrees on $[n]$ is $p_n$, where
\[
P(x) \ := \ \sum_{n \geq 1} p_n \frac{x^n}{n!} \ = \ \frac12 W(-x) - \frac14 W(-x)^2 -\frac12
\log(1+W(-x)) \, .
\]

\item
The number of signed trees on $[n]$ is $st_n = 2^{n-1}n^{n-2}$. The exponential generating function for this sequence is
\[
ST(x) \ := \ \sum_{n \geq 1} 2^{n-1}n^{n-2} \frac{x^n}{n!} \ = \ -\frac12W(-2x) - \frac14 W(-2x)^2.
\]
\item
The number of signed pseudotrees on $[n]$ is $sp_n$, where
\[
SP(x) \ := \ \sum_{n \geq 1} sp_n \frac{x^n}{n!} \ = \ \frac14 W(-2x)  - \log(1+W(-2x)) \, . \]
\item
The number of signed half-edge trees on $[n]$ and of signed loop-trees is $sh_n = sl_n = (2n)^{n-1}$. The exponential generating function for this sequence is
\[
SH(x) \ = \ SL(x) \ := \ \sum_{n \geq 1} (2n)^{n-1} \frac{x^n}{n!} \ = \ -\frac12W(-2x) \, .
\]
\end{enumerate}
\end{proposition}

\begin{proof} We begin by remarking that most of these formulas were obtained by Vladeta Jovovic and posted without proof in entries 
A000272, 
A057500,
A097629,
A320064, and
A052746
of the Online Encyclopedia of Integer Sequences \cite{sloaneonlineseq}. For completeness, we provide proofs.

\smallskip

1. The formula for $t_n$ is well known and due to Cayley; see for example \cite[Proposition 5.3.2]{stanleyec2}. Now, by the multiplicative formula for exponential generating functions \cite[Proposition 5.1.1]{stanleyec2}, $W(-x)^2/2$ is the generating function for pairs of rooted trees $(T_1, r_1)$ and $(T_2,r_2)$, the disjoint union of whose vertex sets is $[n]$. By adding an edge between $r_1$ and $r_2$, we see that this is equivalent to having a single tree with a special chosen edge $r_1r_2$; there are $n^{n-2}(n-1)$ such objects. Therefore
\[
\frac12W(-x)^2 \ = \ \sum_{n \geq 0} n^{n-2}(n-1) \frac{x^n}{n!} \ = \ -W(-x) -T(x) \, ,
\]
proving the desired generating function.

\smallskip

2. A pseudotree on $[n]$ is equivalent to a choice of rooted trees $(T_1, r_1), \ldots, (T_k, r_k)$, the union of whose vertex sets is $[n]$, together with a choice of an undirected cyclic order on $r_1, \ldots, r_n$ --- or equivalently, an undirected cyclic order on those trees. Since the exponential function for rooted trees and for undirected cyclic orders are $-W(-x)$ and 
\[
x + \frac{x^2}2 + \sum_{n \geq 3} \frac{(n-1)!}2 \frac{x^n}{n!} \ = \ \frac{x}2 + \frac{x^2}4
+ \frac12 \log(1-x) \, ,
\]
respectively, the desired result follows by the compositional formula for exponential generating functions.

 \smallskip

3. There are $2^{n-1}$ choices of signs for a tree on $[n]$, so we have $st_n = 2^{n-1}t_n$.
Combining with 1.\ gives the desired formulas.

\smallskip

4. Each pseudotree on $[n]$ can be given $2^n$ different edge sign patterns, half of which will lead to an unbalanced cycle; this leads to $2^{n-1}p_n$ signed pseudotrees. This accounts for all signed pseudotrees, except for the
ones containing a $2$-cycle. We obtain such an object by starting with a signed tree,
choosing one of its edges, and inserting the same edge with the opposite sign. This counts
each such object twice, so the total number of them is $st_n(n-1)/2$. It follows that $sp_n =
2^{n-1}p_n + st_n(n-1)/2$, from which the desired formulas follow using 2.\ and 3. 

\smallskip

5. A signed half-edge tree (or a signed loop-tree) is obtained from a signed tree by choosing the vertex where we will attach the half-edge (or loop). Thus $sh_n = sl_n = n \cdot st_n = (2n)^{n-1}$. The exponential generating function follows directly from the definition of $W(x)$.  
\end{proof}


\subsection{Generating functions of Ehrhart (quasi)polynomials of Coxeter permutahedra}

\begin{theorem} \label{th:egfZono}
The generating functions for the Ehrhart polynomials of the \textbf{integral} Coxeter permutahedra of the classical root systems are:
\begin{eqnarray*}
\sum_{n \geq 0} \ehr_{ \Pi^\ZZ(A_{n-1})} (t) \frac{x^n}{n!} &=&
\exp \left(-\frac1t W(-tx) - \frac1{2t} W(-tx)^2\right) , \\
\sum_{n \geq 0} \ehr_{ \Pi^\ZZ(B_n)} (t) \frac{x^n}{n!} &=&
\left.  \exp\left(- \frac1{2t} \, W(-2tx) - \frac1{4t} W(-2tx)^2\right) \middle/
\sqrt{1+W(-2tx)} \right. , \\
\sum_{n \geq 0} \ehr_{ \Pi^\ZZ(C_n)} (t) \frac{x^n}{n!} &=&
\left.  \exp\left(\frac{-t-1}{2t} \, W(-2tx) - \frac1{4t} W(-2tx)^2\right) \middle/
\sqrt{1+W(-2tx)} \right. , \\
\sum_{n \geq 0} \ehr_{ \Pi^\ZZ(D_n)} (t) \frac{x^n}{n!} &=&
\left.  \exp\left(\frac{t-1}{2t} \, W(-2tx) - \frac1{4t} W(-2tx)^2\right) \middle/
\sqrt{1+W(-2tx)} \right. .
\end{eqnarray*}
\end{theorem}

\begin{proof}
Theorem \ref{th:Ehrhart}.1 tells us that these exponential generating functions can be understood as enumerating various families of (pseudo)forests, weighted by their various types of connected components. The compositional formula for exponential generating functions \cite[Theorem 5.1.4]{stanleyec2} then expresses them in terms of the exponential generating functions for each type of connected component.

For example, in type $A$ there are only tree components, so 

\begin{eqnarray*}
\sum_{n \geq 0} \ehr_{ \Pi^\ZZ(A_{n-1})} (t) \frac{x^n}{n!} 
&=& \sum_{n \geq 0} \, \sum_{\substack{\textrm{ forests} \\ G \text{ on } [n]}}t^{ n - \tc(G) } \frac{x^n}{n!} \\ 
&=& \sum_{n \geq 0} \, \sum_{\substack{\textrm{ forests} \\ G \text{ on } [n]}}\left(\frac1t\right)^{\tc(G)} \frac{(tx)^n}{n!} \\ 
&=& \exp \left(\frac1t \sum_{n \geq 0} \, \sum_{\substack{\textrm{ trees} \\ T \text{ on } [n]}} \frac{(tx)^n}{n!}  \right)\\ 
&=& \exp \left(\frac1t T(tx) \right) \\
&=& \exp \left(-\frac1t W(-tx) - \frac1{2t} W(-tx)^2\right)  
\end{eqnarray*}
by Proposition \ref{prop:trees}.1.

Similarly, for the other types we have

\begin{eqnarray*}
\sum_{n \geq 0} \ehr_{ \Pi^\ZZ(B_n)} (t) \frac{x^n}{n!} 
&=& \sum_{n \geq 0} \, \sum_{\substack{B-\textrm{forests} \\ G \text{ on } [n]}} 2^{\pc(G)} t^{ n - \tc(G) } \frac{x^n}{n!} \\ 
&=& \sum_{n \geq 0} \, \sum_{\substack{B-\textrm{forests} \\ G \text{ on } [n]}} 2^{\pc(G)} \left(\frac1t\right)^{\tc(G)} 1^{\hc(G)} \frac{(tx)^n}{n!} \\
&=& \exp \left(2SP(tx) + \frac1t ST(tx) + SH(tx)\right) 
\end{eqnarray*}

and, analogously,

\begin{eqnarray*}
\sum_{n \geq 0} \ehr_{ \Pi^\ZZ(C_n)} (t) \frac{x^n}{n!} 
&=& \exp \left(2SP(tx) + \frac1t ST(tx) + 2SL(tx)\right),  \\
\sum_{n \geq 0} \ehr_{ \Pi^\ZZ(D_n)} (t) \frac{x^n}{n!} 
&=& \exp \left(2SP(tx) + \frac1t ST(tx) \right).
\end{eqnarray*}
Carefully substituting the formulas in Proposition \ref{prop:trees}, we obtain the desired results.
\end{proof}

\medskip

Using the formulas in Theorem \ref{th:egfZono} and suitable mathematical software, one easily computes the following table of Ehrhart polynomials. The reader may find it instructive to compare this with the analogous table in \cite[Section 6]{ardilacastillohenley}, which lists the Ehrhart polynomials with respect to the weight lattice of each root system.
The tables coincide only in type $C$, which is the only classical type where the weight lattice is $\ZZ^n$.

\begin{table}[h]
\begin{center}
\noindent \begin{tabular}{|c|l|} 
\hline
$\Phi$ &  Ehrhart polynomial of $\Pi^\ZZ(\Phi^+)$ \\ \hline
$A_1$	& $1$	\\
$A_2$	& $1+t$	\\
$A_3$	& $1+3 t+3 t^2$\\
$A_4$	& $1+6 t+15 t^2+16 t^3$\\
\hline
$B_1$	& $1+t$	\\
$B_2$	&$1+4 t+7 t^2$\\
$B_3$	&$1+9 t+39 t^2+87 t^3$\\
$B_4$	&$1+16 t+126 t^2+608 t^3+1553 t^4$\\
\hline
$C_1$	&$1+2t$	\\
$C_2$	&$1+6 t+14 t^2$\\
$C_3$	&$1+12 t+66 t^2+172 t^3$\\
$C_4$	&$1+20 t+192 t^2+1080 t^3+3036 t^4$\\
\hline
$D_2$	&$1+2t+2 t^2$\\
$D_3$	&$1+6 t+18 t^2+32 t^3$\\
$D_4$	&$1+12 t+72 t^2+280 t^3+636 t^4$\\
\hline
\end{tabular}
\end{center}
\caption{\label{table:polys}Ehrhart polynomials of integral Coxeter permutahedra.}
\end{table}

\begin{theorem} \label{th:egfCZono}
The generating function for the odd part of the Ehrhart quasipolynomials of the non-integral \textbf{standard} Coxeter permutahedra are the following. For $t$ odd,
\begin{eqnarray*}
\sum_{n \geq 0} \ehr_{ \Pi(A_{2n-1})} (t) \frac{x^{2n}}{(2n)!} 
&=&
\exp \left(-\frac{W(-tx)+W(tx)}{2t} - \frac{W(-tx)^2+W(tx)^2}{4t}\right)  \\
\sum_{n \geq 0} \ehr_{ \Pi(B_n)} (t) \frac{x^n}{n!} &=&
\left.  \exp\left(- \frac{W(-2tx) + W(2tx)}{4t} - \frac{W(-2tx)^2+W(2tx)^2}{8t}\right)
\middle/ \sqrt{1+W(-2tx)} \right. .
\end{eqnarray*}
\end{theorem}

\begin{proof}
We carry out similar computations as for Theorem \ref{th:egfZono}. This requires us to observe that the generating functions for even trees and even signed trees are
\begin{eqnarray*}
T_{\text{even}}(x) &:=& \sum_{n \geq 0} t_{2n}\frac{x^{2n}}{n!} \ = \ \frac12(T(x) + T(-x)),\\
ST_{\text{even}}(x) &:=& \sum_{n \geq 0} st_{2n}\frac{x^{2n}}{n!} \ = \ \frac12(ST(x) +
ST(-x)) \, .
\end{eqnarray*}
Now, in light of Theorem \ref{th:Ehrhart}.2, and analogously to the proof of Theorem \ref{th:egfZono}, we have
\begin{eqnarray*}
\sum_{n \geq 0} \ehr_{ \Pi(A_{2n-1})} (t) \frac{x^{2n}}{(2n)!} 
&=& \exp \left(\frac1tT_{\text{even}}(tx)\right) \\
&=& \exp \left(\frac1{2t}T(tx) + \frac1{2t}T(-tx)\right)
\end{eqnarray*}
and 
\begin{eqnarray*}
\sum_{n \geq 0} \ehr_{ \Pi(B_n)} (t) \frac{x^n}{n!} 
&=& \exp \left(2SP(tx) + \frac1t ST_{\text{even}}(tx) +  2SL(tx)\right) \\
&=& \exp \left(2SP(tx) + \frac1{2t} ST(tx) + \frac1{2t} ST(-tx) + 2SL(tx)\right),  
\end{eqnarray*}
which give the desired results using Proposition~\ref{prop:trees}.
\end{proof}

Using these formulas, and combining them with Table \ref{table:polys}, one computes the following table of Ehrhart quasipolynomials. 

\begin{table}[h]
\begin{center}
\noindent \begin{tabular}{|c|l|}
\hline
$\Phi$ & Ehrhart quasipolynomial of $\Pi(\Phi^+)$ \\ \hline
$A_2$	& 
$\begin{cases}
1+t & \text{for $t$ even} \\
t & \text{for $t$ odd} \\
\end{cases}
$
\\
$A_4$	& 
$\begin{cases}
1+6t+15t^2+16t^3 & \text{for $t$ even} \\
3t^2+16 t^3 & \text{for $t$ odd} \\
\end{cases}
$
\\
\hline
$B_1$	&
$
\begin{cases}
1+t & \text{for $t$ even} \\
t & \text{for $t$ odd} \\
\end{cases}
$
\\
$B_2$	&
$
\begin{cases}
1+4t + 7t^2 & \text{for $t$ even} \\
2 t+7 t^2 & \text{for $t$ odd} \\
\end{cases}
$
\\
$B_3$	&
$
\begin{cases}
1+9t + 39t^2 + 87t^3 & \text{for $t$ even} \\
6 t^2+87 t^3 & \text{for $t$ odd} \\
\end{cases}
$\\
$B_4$	&
$
\begin{cases}
1+16t + 126t^2 + 608t^3 + 1553t^4 & \text{for $t$ even} \\
12 t^2+212 t^3+1553 t^4 & \text{for $t$ odd} \\
\end{cases}
$\\
\hline
\end{tabular}
\end{center}
\caption{\label{table:quasipolys}Ehrhart quasipolynomials of the non-integral standard Coxeter permutahedra.}
\end{table}

The reader may find it instructive to count the lattice points in the polygons of Figure \ref{root permutahedra}, and compare those numbers with the predictions given by Tables \ref{table:polys} and~\ref{table:quasipolys}.


\section{Acknowledgments}

Some of the results of this paper are part of the Master's theses of JM at San Francisco State University, under the supervision of FA and MB \cite{McWhirter}.
We would like to thank Mariel Supina and Andr\'es Vindas--Mel\'endez for valuable discussions, and Jean-Philippe Labb\'e for checking our computations of the Ehrhart (quasi)polynomials of Tables \ref{table:polys} and \ref{table:quasipolys}.
This paper was written while FA was on sabbatical at the Universidad  de Los Andes in Bogot\'a. He thanks Los Andes for their hospitality and SFSU and the Simons Foundation for their financial support.


\newpage

\small

\bibliographystyle{amsplain}
\bibliography{bib}

\setlength{\parskip}{0cm} 

\end{document}